\documentclass[11pt]{amsart}
\usepackage{colonequals, amssymb}
\usepackage[alphabetic]{amsrefs}
\usepackage{enumitem}
\usepackage{fullpage}
\usepackage[colorlinks,anchorcolor=blue,citecolor=blue,linkcolor=blue,urlcolor =blue,bookmarksopenlevel=2,bookmarksopen=true]{hyperref}
\AtBeginDocument{%
	\def\MR#1{}
}

\newtheorem{theorem}{Theorem}[section]
\newtheorem{lemma}[theorem]{Lemma}
\newtheorem{proposition}[theorem]{Proposition}

\allowdisplaybreaks

\theoremstyle{definition}

\newtheorem{notation}[theorem]{Notation}
\newtheorem{question}[theorem]{Question}

\theoremstyle{remark}
\newtheorem{remark}[theorem]{Remark}

\numberwithin{equation}{section}

\newcommand{\bP}{\operatorname{\mathbb{P}}}

\newcommand{\bF}{\operatorname{\mathbb{F}}}

\newcommand{\homog}{\operatorname{homog}}
\newcommand{\spec}{\operatorname{Spec}}

\numberwithin{equation}{section}

\makeatletter
\@namedef{subjclassname@2020}{%
  \textup{2020} Mathematics Subject Classification}
  \makeatother

\begin{document}

\title{On the proportion of transverse-free plane curves}

\author{Shamil Asgarli}
\address{Department of Mathematics, University of British Columbia, Vancouver, BC V6T 1Z2}
\email{sasgarli@math.ubc.ca}

\author{Brian Freidin}
\address{Department of Mathematics, University of British Columbia, Vancouver, BC V6T 1Z2}
\email{bfreidin@math.ubc.ca}


\maketitle

\begin{abstract} We study the asymptotic proportion of smooth plane curves over a finite field $\bF_q$ which are tangent to every line defined over $\bF_q$. This partially answers a question raised by Charles Favre. Our techniques include applications of Poonen's Bertini theorem and Schrijver's theorem on perfect matchings in regular bipartite graphs. Our main theorem implies that a random smooth plane curve over $\bF_q$ admits a transverse $\bF_q$-line with very high probability.
\end{abstract}

\section{Introduction}

Let $C\subset\bP^2$ be a plane curve defined over an arbitrary field $k$. If $k$ is an infinite field, then we can apply Bertini's theorem to find a line $L$ defined over $k$ such that $L$ meets $C$ transversely. By B\'ezout's theorem, any line (which is not a component of the curve $C$) meets $C$ in exactly $d=\deg(C)$ points with coordinates in $\overline{k}$ counted with multiplicity; by definition, a \textbf{transverse line} is one which meets $C$ in $d$ \textit{distinct} points. Equivalently, $L$ is \emph{not} transverse to $C$ if and only if $L$ is tangent to $C$ or $L$ passes through a singular point of $C$.

When $k=\bF_q$ is a finite field, it is not necessarily true that we can find a line $L$ defined over $\bF_q$ such that $L$ meets $C$ transversely. It is possible that each of the $q^2+q+1$ lines in $\bP^2$ defined over $\bF_q$ is tangent to $C$. Indeed, there exist smooth plane curves of degree $q+2$ with this property \cite{Asg19}*{Example 2.A}. There are also smooth examples in each sufficiently large degree \cite{Asg19}*{Example 2.B}.

Given a plane curve $C\subset\bP^2$ defined over a finite field $\bF_q$, we say that $C$ is \textbf{transverse-free} if $C$ admits no transverse line defined over $\bF_q$.

In the present paper, we will restrict our attention to smooth plane curves. In order to count the proportion of smooth transverse-free plane curves, we use the notion of a \textbf{natural density} denoted by $\mu$. Given a subset of $\mathcal{A} \subset R\colonequals\bF_q[x,y,z]$, the natural density $\mu(\mathcal{A})$ is the limit as $d\to\infty$ of the fraction of elements of degree $d$ in $\mathcal{A}$ among all degree $d$ homogeneous polynomials in $R$. Note that counting homogeneous polynomials and curves of the same degree differ by a factor arising from scaling, so that the limiting proportion for both counts coincide.  Our main result is the following.

\begin{theorem}\label{thm:smooth-proportion-transverse-free}
Let $\mathcal{F}\subset\bF_{q}[x,y,z]$ be the set of polynomials defining smooth transverse-free plane curves. Then
\begin{align*}
e^{-(q^2+q+1)}\cdot \left(\frac{1}{q}+\frac{1}{q^2}-\frac{1}{q^3}\right)^{q^2+q+1} \leq \mu(\mathcal{F}) \leq 7.5\cdot \left(\frac{1}{q}+\frac{1}{q^2}-\frac{1}{q^3}\right)^{q^2+q+1} 
\end{align*}
\end{theorem}
where $e=2.71828...$ is Euler's number.

\medskip

\noindent\textbf{Acknowledgments.} We are grateful to Charles Favre who first raised the question of computing the natural density of transverse-free plane curves over finite fields. C. Favre asked for the exact value of $\mu(\mathcal{F})$, and Theorem~\ref{thm:smooth-proportion-transverse-free} is our partial answer. We also thank Dori Bejleri and Zinovy Reichstein for helpful conversations on this topic. Both authors were supported by postdoctoral research fellowships from the University of British Columbia. The second author was partially supported by a postdoctoral fellowship from the Pacific Institute for the Mathematical Sciences.

\section{Overview of the paper}\label{sect:notation}

In this section, we discuss notation used in the paper, interpret our main theorem, and explain how the paper is organized.

We will work with the finite fields $\bF_{q}$ where $q$ is a fixed prime power. We set $R\colonequals \bF_q[x,y,z]$ and denote by $R_{d}$ the vector space spanned by degree $d$ homogeneous polynomials in $R$. By convention, $0$ is considered a homogeneous polynomial in each degree. As a shorthand, we also define
\begin{align*}
    R_{\homog} = \bigcup_{d\geq 1} R_{d}
\end{align*}
to be the set of non-constant homogeneous polynomials. By definition, elements of $R_{\text{homog}}$ define projective plane curves in $\bP^2$. Elements of $R_{1}$ correspond to lines defined over $\bF_q$, also referred to as  $\bF_q$-lines. Similarly, the points in $\bP^2(\bF_q)$ will be called $\bF_q$-points.

As mentioned in the introduction, we can ask for the asymptotic proportion of a given subset of the ring $R$. More precisely, given a subset $\mathcal{A}\subset R$, we first define $\mathcal{A}_d=\mathcal{A}\cap R_d$ and $\mu_d(\mathcal{A})= \frac{\# \mathcal{A}_d}{\# R_d}$ for each $d\geq 1$. The \textbf{natural density} of $\mathcal{A}$ is defined by
\begin{align*}
    \mu(\mathcal{A}) = \lim_{d\to\infty} \mu_d(\mathcal{A}) = \lim_{d\to\infty} \frac{\# \mathcal{A}_d}{\# R_d}
\end{align*}
provided that the limit exists. Similarly, we define
\begin{align*}
    \underline{\mu}(\mathcal{A}) = \liminf_{d\to\infty} \frac{\# \mathcal{A}_d}{\# R_d} \ \ \ \text{and} \ \ \ 
     \overline{\mu}(\mathcal{A}) = \limsup_{d\to\infty} \frac{\# \mathcal{A}_d}{\# R_d}
\end{align*}
as \textbf{lower} and \textbf{upper densities}, respectively. If upper and lower densities agree, then the natural density exists, and we have the equality $\underline{\mu}(\mathcal{A})=\overline{\mu}(\mathcal{A})=\mu(\mathcal{A})$. As an example with geometric flavour, let $R_{\text{smooth}}\subset R$ denote the subset of polynomials defining smooth plane curves. As a special case of Poonen's Theorem \cite{Poo04}*{Theorem 1.1}, we obtain
\begin{align}\label{eq:probability-smooth-curve}
\mu(R_{\text{smooth}}) = \zeta_{\bP^2}(3)^{-1} = \left(1-\frac{1}{q}\right)\left(1-\frac{1}{q^2}\right)\left(1-\frac{1}{q^3}\right).
\end{align}
Thus, the quantity above can interpreted as the probability that a random plane curve defined over $\bF_q$ is smooth. As mentioned in the introduction, we are counting homogeneous polynomials of degree $d$ instead of plane curves of degree $d$. The limiting proportions are the same, because
\begin{align*}
   \frac{\# \text{ plane curves defined by polynomials of degree } d \text { in } \mathcal{A}}{\# \text{ all plane curves of degree } d}&=  \frac{(\#\mathcal{A}_d-1)/(q-1)}{(\# R_d-1)/(q-1)} \\
   &=\frac{\#\mathcal{A}_d - 1}{\# R_d -1},
\end{align*}
which agrees with $\mu_d(\mathcal{A})$ in the limit as $d\to\infty$.

The aim of the present paper is to investigate the density of the subset
\begin{align*}
\mathcal{F}\colonequals \{f\in R_{\homog} \ | \ & C=\{f=0\} \text{ is a smooth plane curve} \\ 
&\text{such that all $\bF_q$-lines are tangent to $C$}\}.
\end{align*}
A smooth plane curve $C=\{f=0\}$ does not admit any transverse $\bF_q$-lines if and only if $f\in \mathcal{F}$. Hence, $\mathcal{F}$ corresponds to precisely the collection of smooth transverse-free plane curves over $\bF_q$.

We remark that a smooth plane curve $C$ is transverse-free if and only if the dual curve $C^{\ast}$ is space-filling in the sense that $C^{\ast}$ passes through all $\bF_q$-points of the dual projective plane $(\bP^2)^{\ast}$. Indeed, the dual curve $C^{\ast}$ is the plane curve whose points correspond to tangent lines to $C$, and the condition that $C$ is transverse-free is equivalent to asserting that each $\bF_q$-line $L$ is tangent to $C$, that is, $C^{\ast}(\bF_q)=(\bP^2)^{\ast}(\bF_q)$. With this remark in mind, we can also define the set $\mathcal{F}$ as
\begin{align*}
    \mathcal{F}=\{f\in R_{\homog} \ | \ C=\{f=0\} \text{ is smooth and } C^{\ast} \text{ is space-filling}\}.
\end{align*}
It is not clear that $\mu(\mathcal{F})$ actually exists as a limit. We will prove that this is the case in Section~\ref{sect:mu-exists}.

Let us interpret the upper bound given in Theorem~\ref{thm:smooth-proportion-transverse-free} as a probabilistic version of the Bertini's theorem. Using the notation of conditional probability, we have
\begin{align*}
    \mathbb{P}\left(C \text{ is tangent to each } \bF_q\text{-line} \ | \ C \text{ is smooth} \right)
     &= \frac{\mathbb{P}\left(C \text{ is smooth, and tangent to each } \bF_q\text{-line}\right)}{\mathbb{P}\left(C \text{ is smooth}\right)} \\
     &=\frac{\mu(\mathcal{F})}{\mu(R_{\operatorname{smooth}})} \\
     &\leq \frac{7.5\left(q^{-1}+q^{-2}-q^{-3}\right)^{q^2+q+1}}{(1-q^{-1})(1-q^{-2})(1-q^{-3})}
\end{align*}
using the upper bound for $\mu(\mathcal{F})$ in Theorem~\ref{thm:smooth-proportion-transverse-free} and \eqref{eq:probability-smooth-curve}. Thus the probability that a random smooth plane curve $C$ defined over $\bF_q$ satisfies the conclusion of Bertini's theorem (that is, $C$ admits a transverse $\bF_q$-line) can be bounded below.
\begin{align}\label{eq:probability-Bertini-satisfied}
 \mathbb{P}\left(C \text{ admits a transverse } \bF_q\text{-line} \ | \ C \text{ is smooth} \right) \geq 1-\frac{7.5\left(q^{-1}+q^{-2}-q^{-3}\right)^{q^2+q+1}}{(1-q^{-1})(1-q^{-2})(1-q^{-3})}.
\end{align}
The lower bound in \eqref{eq:probability-Bertini-satisfied} approaches $1$ as $q$ gets larger, and it is already very close to $1$ even for small values of $q$. For example, for $q=3$
\begin{align*}
1-\frac{7.5\left(3^{-1}+3^{-2}-3^{-3}\right)^{3^2+3+1}}{(1-3^{-1})(1-3^{-2})(1-3^{-3})}= 0.99988803...
\end{align*}
In other words, a random smooth plane curve over $\bF_3$ admits a transverse $\bF_3$-line with probability at least $0.9998$. For larger values of $q$, this probability is even closer to $1$. However, for $q=2$, the bound \eqref{eq:probability-Bertini-satisfied} only gives about $0.1485$ as a lower bound for the probability of a random smooth curve over $\bF_2$ to admit a transverse $\bF_2$-line.

The paper is organized as follows. In Section~\ref{section:lower-bound}, we obtain the lower bound for $\underline{\mu}(\mathcal{F})$, which will prove the first half of Theorem~\ref{thm:smooth-proportion-transverse-free} once we show $\mu(\mathcal{F})$ exists. The key inputs here are Poonen's Bertini theorem \cite{Poo04} and Schrijver's lower bound on the number of perfect matchings in a regular bipartite graph \cite{Sch98}. In Section~\ref{section:independence}, we compute the density of curves with various singularity and tangency conditions. This section provides many examples of sets $\mathcal{A}$ where the natural density $\mu(\mathcal{A})$ can be computed explicitly. In addition, we show in Lemma~\ref{lemma:super-independence} that the density of an intersection can often be computed by taking a product of densities. In Section~\ref{section:upper-bound}, we borrow the results developed in Section~\ref{section:independence} to obtain the upper bound for $\overline{\mu}(\mathcal{F})$. The proof of the upper bound is achieved via reduction to the case where the tangency points have bounded degree, together with a combinatorial technique to partition the desired set based on the locations of tangency points. Section~\ref{sect:mu-exists} is devoted to the proof that the natural density of $\mathcal{F}$ exists.

\section{Proof of the lower bound }\label{section:lower-bound}

The objective of this section is to prove the following.

\begin{theorem}\label{thm:lower-bound} Let $\mathcal{F}\subset R=\bF_q[x,y,z]$ denote the subset defining smooth transverse-free plane curves. Then
\begin{align*}
    \underline{\mu}(\mathcal{F})\geq e^{-(q^2+q+1)} \cdot \left(\frac{1}{q}+\frac{1}{q^2}-\frac{1}{q^3}\right)^{q^2+q+1}.
\end{align*}
\end{theorem}

Our task is to produce a positive proportion of smooth transverse-free curves of degree $d$. Recall that a smooth plane curve $C$ is transverse-free if each of the $q^2+q+1$ lines defined over $\bF_q$ is tangent to $C$ at some geometric point (i.e. a point defined over $\overline{\bF}_q$). By requiring each of the tangency points to be an $\bF_q$-point of $C$, we will already obtain a required positive proportion. 

Our strategy consists of two steps:

\begin{enumerate}
    \item\label{strategy:permanents} We can order $\bF_q$-points $\{P_i\}_{i=1}^{q^2+q+1}$  and $\bF_q$-lines $\{L_i\}_{i=1}^{q^2+q+1}$ in $\bP^2$ so that $P_i\in L_i$ for each $1\leq i\leq q^2+q+1$.
    \item\label{strategy:poonen} We use \cite{Poo04}*{Theorem 1.2} to compute the natural density of smooth plane curves whose tangent line at $P_i$ is $L_i$ for each $1\leq i\leq q^2+q+1$. These curves are transverse-free by construction.
\end{enumerate}
Section~\ref{subsect:perm} is devoted to understanding (\ref{strategy:permanents}), which is a purely combinatorial statement about $\bP^2(\bF_q)$. In Section~\ref{subsect:constructing-transverse-free-curves}, we explain how to carry out the procedure described in (\ref{strategy:poonen}). 

\subsection{Permanent of the projective plane.}\label{subsect:perm} An enumeration of $\bF_q$-points $\{P_1, \ldots, P_{q^2+q+1}\}$ and $\bF_q$-lines $\{L_1, \ldots, L_{q^2+q+1}\}$ such that $P_i\in L_i$ for each $1\leq i\leq q^2+q+1$ will be called a \textbf{perfect point-line correspondence}. One can use Hall's matching theorem to prove that such a correspondence exists~\cite{CR04}*{Theorem 1.1}. More generally, in the paper~\cite{CR04}, De Concini and Reichstein study nesting maps of Grassmannians of complementary dimensional subspaces.

In order to produce many transverse-free smooth plane curves, we are interested in constructing many different perfect point-line correspondences. We are naturally lead to the following purely combinatorial question.

\begin{question}\label{quest:point-line}
How many perfect point-line correspondences are there in $\bP^2(\bF_q)$?
\end{question}

We can phrase Question~\ref{quest:point-line} in terms of perfect matchings of a certain bipartite graph known as the \textbf{Levi graph} or \textbf{incidence graph}. The Levi graph of $\bP^2(\bF_q)$ is the bipartite graph whose vertices correspond to the $q^2+q+1$ points and $q^2+q+1$ lines of $\bP^2(\bF_q)$, and vertices corresponding to point $P$ and line $L$ are adjacent if $P\in L$. It follows that a perfect point-line correspondence is precisely the same as a perfect matching in the Levi graph. Note that the Levi graph is regular since each vertex has degree $q+1$.

The number of perfect matchings in a bipartite graph can be computed by taking the permanent of the incidence matrix. Recall that the permanent of an $n\times n$ square matrix $M=(m_{ij})$ is 
\begin{align*}
\operatorname{per}(M)\colonequals \sum_{\sigma\in S_n} \prod_{i=1}^{n} m_{i\sigma(i)}.
\end{align*} 
In the case of the Levi graph, the incidence matrix $M_{q^2+q+1}$ is a $(q^2+q+1)\times (q^2+q+1)$ matrix where the rows are indexed by points and columns are indexed by lines of $\bP^2(\bF_q)$. The $(i, j)$-th entry of $M_{q^2+q+1}$ is $0$ or $1$, depending on whether the $i$-th point lies on the $j$-th line. It follows that the permanent $\operatorname{per}(M_{q^2+q+1})$ exactly counts the number of perfect point-line correspondences in $\bP^2(\bF_q)$. We refer to $\operatorname{per}(M_{q^2+q+1})$ as the \textbf{permanent of the projective plane of order $q$}.

One can calculate these permanents for $q=2, 3, 4$ and $5$ in a computer, and the results are:
\begin{align*}
    \operatorname{per}(M_{7})&=24 \\  
    \operatorname{per}(M_{13})&=3852 \\
    \operatorname{per}(M_{21})&=18534400 \\
    \operatorname{per}(M_{31})&=4598378639550.
\end{align*}
This sequence appears as A000794 in Sloane's on-line encyclopedia of integer sequences \cite{OEIS} under the name of ``Permanent of projective plane of order $n$". To our knowledge, these permanents beyond the case $q=5$ have not been calculated explicitly.

While the exact answer to Question~\ref{quest:point-line} seems to be known only for $q\in\{2,3,4,5\}$, we can provide the following lower bound that works for all $q\geq 2$. 

\begin{proposition}\label{prop:lower-bound-permanent} Let $q$ be a prime power. The permanent of the projective plane of order $q$ is at least 
\begin{align*}
    \left(\frac{q+1}{e}\right)^{q^2+q+1}.
\end{align*}
In particular, this is a lower bound on the number of perfect point-line correspondences in $\bP^2(\bF_q)$.
\end{proposition}

Before we discuss the proof, we recall van der Waerden's conjecture on permanents. A matrix is called \textbf{doubly-stochastic} if it has non-negative entries and each row and column sums up to $1$. Van der Waerden's conjecture was that if $A$ is an $n\times n$ doubly-stochastic matrix then $\operatorname{per}(A)\geq n! / n^n$. The conjecture was proved independently by Gyires~\cite{Gyi80}, Egorychev~\cite{Ego81} and Falikman~\cite{Fal81}. Almost three decades later, Gurvitz \cite{Gur08} gave a simpler proof, and his proof technique was explored further by Laurent and Schrijver \cite{MS10}.  In particular,  Laurent and Schrijver gave a short proof of the following theorem.

\begin{theorem}\label{thm:schrijver} Suppose that $A$ is an $n\times n$ integer matrix with non-negative entries. If all row and column sums of $A$ are equal to $k$, then
$$
\operatorname{per}(A) \geq \left(\frac{(k-1)^{k-1}}{k^{k-2}}\right)^{n}.
$$
\end{theorem}

The above theorem and its proof appear in \cite{MS10}*{Corollary 1C}. We remark that Theorem~\ref{thm:schrijver} was first proved by Schrijver~\cite{Sch98} using a different argument.

\begin{proof}[Proof of Proposition~\ref{prop:lower-bound-permanent}] Let $M_{q^2+q+1}$ be the incidence matrix of the projective plane of order $q$. Since each point is contained in $q+1$ lines, and each line contains $q+1$ points, it follows that all row and column sums of $M_{q^2+q+1}$ equal to $q+1$. By applying Theorem~\ref{thm:schrijver}, we obtain the lower bound
\begin{align*}
\operatorname{per}(M_{q^2+q+1}) \geq \left(\frac{q^{q}}{(q+1)^{q-1}}\right)^{q^2+q+1} &= (q+1)^{q^2+q+1} \left(\frac{q}{q+1}\right)^{q(q^2+q+1)} \\
&\geq \left(\frac{q+1}{e}\right)^{q^2+q+1} 
\end{align*}
where we used the fact that $\left(\frac{q}{q+1}\right)^{q}\geq e^{-1}$. \end{proof}

\subsection{Constructing transverse-free curves.}\label{subsect:constructing-transverse-free-curves} The Proposition~\ref{prop:lower-bound-permanent} provides us with plenty of perfect point-line correspondences in $\bP^2(\bF_q)$. As explained in the beginning of the section, each perfect point-line correspondence gives rise to a positive proportion of smooth transverse-free plane curves. 

\begin{lemma}\label{lemma:lower-bound-for-single-point-line}
Given a perfect point-line correspondence $\{P_i\}_{i=1}^{q^2+q+1}$ and $\{L_i\}_{i=1}^{q^2+q+1}$, the natural density of smooth plane curves which are tangent to $L_i$ at the point $P_i$ is:
\begin{align*}
\left(1-\frac{1}{q}\right)\left(1-\frac{1}{q^2}\right)\left(1-\frac{1}{q^3}\right) q^{-2(q^2+q+1)}.
\end{align*}
\end{lemma}

\begin{proof} We use Poonen's Bertini Theorem with Taylor conditions \cite{Poo04}*{Theorem 1.2}, which computes the density of smooth plane curves over $\bF_q$ with any given finitely many tangency conditions. Given $P_i=[a_i:b_i:c_i]$, consider the subscheme $Z_i\subset\mathbb{P}^2$ defined by
\begin{align*}
    Z_i = \spec\left( \frac{\mathbb{F}_q[x,y,z]}{(c_i x-a_i z, c_i y-b_iz)^2}\right)
\end{align*}
for each $1\leq i\leq q^2+q+1$. In other words, $Z_i$ is a subscheme whose support is the point $P_i=[a_i : b_i : c_i]$, and the non-reduced structure keeps track of the first order tangent information. Next, we consider the subscheme $Z\subset\bP^2$ defined by
\begin{align*}
Z = \coprod_{i=1}^{q^2+q+1} Z_i.
\end{align*}
Note that the global sections $H^{0}(Z, \mathcal{O}_Z)$ can be expressed as a direct sum:
\begin{align*}
    H^0(Z, \mathcal{O}_Z) = \bigoplus_{i=1}^{q^2+q+1} H^{0}(Z_i, \mathcal{O}_{Z_i}).
\end{align*}
Since $H^{0}(Z_i, \mathcal{O}_{Z_i})$ is a $3$-dimensional vector space over $\bF_q$, it follows that $H^{0}(Z, \mathcal{O}_Z)$ has dimension $3(q^2+q+1)$ as an $\bF_q$-vector space. 

For each $1\leq i\leq q^2+q+1$, let $T_i\subset H^0(Z_i, \mathcal{O}_{Z_i})$ denote the linear subspace corresponding to picking the particular line $L_i$ to be the tangent direction at $P_i$. Note that $T_i$ has codimension $2$ inside $H^0(Z_i,\mathcal{O}_{Z_i})$. In particular, $\dim_{\bF_q} T_i=1$. Letting $T=\bigoplus_{i=1}^{q^2+q+1} T_i$, we see that $T\subset H^{0}(Z,\mathcal{O}_Z)$ and $\dim_{\bF_q}T=q^2+q+1$. Finally, consider the set described in the statement of the lemma,
\begin{align*}
\mathcal{P} = \{f\in \bF_q[x,y,z] \ | \ C=\{f=0\} \text{ is smooth, and has tangent line } L_i \text{ at } P_i \}.
\end{align*}
By \cite{Poo04}*{Theorem 1.2}, we have
\begin{align*}
    \mu(\mathcal{P}) = \frac{\#T}{\# H^0(Z, \mathcal{O}_Z)} \zeta_{\bP^2}(3)^{-1} &= \left(\frac{q}{q^3}\right)^{q^2+q+1} \zeta_{\bP^2}(3)^{-1} \\
    &=\left(1-\frac{1}{q}\right)\left(1-\frac{1}{q^2}\right)\left(1-\frac{1}{q^3}\right) q^{-2(q^2+q+1)}
\end{align*}
as desired. \end{proof}

Before we proceed with the proof of Theorem~\ref{thm:lower-bound}, we establish the following technical inequality.

\begin{lemma}\label{lemma:technical-inequality}
For each $q\geq 2$, we have
\begin{align*}
(q+1)^{q^2+q+1}\left(1-\frac{1}{q}\right)\left(1-\frac{1}{q^2}\right)\left(1-\frac{1}{q^3}\right)\geq  q^{2(q^2+q+1)} \left(\frac{1}{q}+\frac{1}{q^2}-\frac{1}{q^3}\right)^{q^2+q+1}.
\end{align*}
\end{lemma}

\begin{proof}
After rearranging the terms, the desired inequality is equivalent to
\begin{align*}
    \left(\frac{q+1}{q}\right)^{q^2+q+1}\left(1-\frac{1}{q}\right)\left(1-\frac{1}{q^2}\right)\left(1-\frac{1}{q^3}\right) \geq \left(1+\frac{1}{q}-\frac{1}{q^2}\right)^{q^2+q+1}.
\end{align*}
After writing $1+\frac{1}{q}-\frac{1}{q^2}=\frac{q^2+q-1}{q^2}$, the previous inequality becomes
\begin{align*}
    \left(\frac{q+1}{q}\right)^{q^2+q+1}\left(1-\frac{1}{q}\right)\left(1-\frac{1}{q^2}\right)\left(1-\frac{1}{q^3}\right) \geq \left(\frac{q^2+q-1}{q^2}\right)^{q^2+q+1}.\end{align*}
After rearranging the terms again, we need to prove that
\begin{align}\label{eq:technical-inequality-key-function}
    h(q)\colonequals \left(1+\frac{1}{q^2+q-1}\right)^{q^2+q+1}\left(1-\frac{1}{q}\right)\left(1-\frac{1}{q^2}\right)\left(1-\frac{1}{q^3}\right) \geq 1.
\end{align} 
Our goal has been reduced to showing that $h(q)\geq 1$ for each $q\geq 2$. For $q=2$, one can check that
\begin{align*}
    h(2)=\left(\frac{6}{5}\right)^7 \left(\frac{1}{2}\right)\left(\frac{3}{4}\right)\left(\frac{7}{8}\right) = \frac{91854}{78125} = 1.1757312 \geq 1 
\end{align*}
as claimed. So, we will assume $q\geq 3$. Using the first two terms of the binomial theorem, we have
\begin{align*}
    h(q)\geq \left(1+\frac{q^2+q+1}{q^2+q-1}\right)\left(1-\frac{1}{q}\right)\left(1-\frac{1}{q^2}\right)\left(1-\frac{1}{q^3}\right).
\end{align*}
To prove $h(q)\geq 1$ for $q\geq 3$, it therefore suffices to prove the stronger inequality,
\begin{align}\label{eq:technical-inequality-intermediate-step}
    \left(1+\frac{q^2+q+1}{q^2+q-1}\right)\left(1-\frac{1}{q}\right)\left(1-\frac{1}{q^2}\right)\left(1-\frac{1}{q^3}\right) \geq 1,
\end{align}
which rearranges into
\begin{align*}
    2(q^2+q)(q-1)(q^2-1)(q^3-1)\geq (q^2+q-1)q^6.
\end{align*}
After simplification, we need to prove
\begin{align*}
    2(q^2-1)^2 (q^3-1) \geq (q^2+q-1)q^5
\end{align*}
for $q\geq 3$. This last inequality is equivalent to
\begin{align}\label{eq:technical-ineq-last-step}
    q^7 - q^6 - 3q^5 - 2 q^4 + 2 q^3 + 4 q^2 - 2 \geq 0.
\end{align}
It is clear that $2 q^3 + 4 q^2 - 2\geq 0$. Since we are assuming $q\geq 3$, we have
\begin{align*}
    q^7 \geq 3q^6 = q^6+2q^6\geq q^6+6q^5.
\end{align*}
Thus,
\begin{align*}
    q^7 - q^6 - 3q^5 - 2 q^4 \geq (q^6+6q^5) - q^6 - 3q^5 - 2 q^4 = 3q^5-2q^4>0.
\end{align*}
This proves the inequality~\eqref{eq:technical-ineq-last-step} for $q\geq 3$. Consequently, \eqref{eq:technical-inequality-intermediate-step} follows, and hence also \eqref{eq:technical-inequality-key-function}. \end{proof}

We are now ready to establish the lower bound for $\underline{\mu}(\mathcal{F})$.

\begin{proof}[Proof of Theorem~\ref{thm:lower-bound}] By Lemma~\ref{lemma:lower-bound-for-single-point-line}, each perfect point-line correspondence gives rise to a collection of smooth transverse-free curves with natural density equal to
\begin{align*}
\left(1-\frac{1}{q}\right)\left(1-\frac{1}{q^2}\right)\left(1-\frac{1}{q^3}\right) q^{-2(q^2+q+1)}.
\end{align*}
Now, there are at least $\left(\frac{q+1}{e}\right)^{q^2+q+1}$ perfect point-line correspondences by Proposition~\ref{prop:lower-bound-permanent}. Since we are considering smooth curves, two different perfect point-line correspondences give rise to disjoint sets of transverse-free curves. Thus
\begin{align}\label{eq:lower-bound-intermediate-step}
\underline{\mu}(\mathcal{F}) &\geq 
\underbrace{\left(\frac{q+1}{e}\right)^{q^2+q+1}}_{\text{number of correspondences}} \cdot  \underbrace{\left(1-\frac{1}{q}\right)\left(1-\frac{1}{q^2}\right)\left(1-\frac{1}{q^3}\right) q^{-2(q^2+q+1)}}_{\text{contribution from each correspondence}}.
\end{align}
By Lemma~\ref{lemma:technical-inequality},
\begin{align}\label{eq:technical-inequality}
    (q+1)^{q^2+q+1}\left(1-\frac{1}{q}\right)\left(1-\frac{1}{q^2}\right)\left(1-\frac{1}{q^3}\right)\geq q^{2(q^2+q+1)} \left(\frac{1}{q}+\frac{1}{q^2}-\frac{1}{q^3}\right)^{q^2+q+1} 
\end{align}
By combining~\eqref{eq:lower-bound-intermediate-step}~and~\eqref{eq:technical-inequality}, we obtain
\begin{align*}
\underline{\mu}(\mathcal{F}) &\geq e^{-(q^2+q+1)} \cdot \left(\frac{1}{q}+\frac{1}{q^2}-\frac{1}{q^3}\right)^{q^2+q+1}
\end{align*}
as desired. 
\end{proof}

\section{Computing natural density with examples}\label{section:independence}

In this section, we investigate the natural densities of various subsets of $R_{\homog}$. These results will be used in Section~\ref{section:upper-bound}. 

The following three subsets and their properties will be the focus of this section. The first few lemmas will compute the densities of these subsets. Given a point $Q\in\bP^2(\bF_q)$, we define
\begin{align*}
    \mathcal{S}_{Q} = \{ f\in R_{\homog} \ | \ C=\{f=0\} \text{ is singular at } Q \}.
\end{align*}
Given a line $L$ and a point $P\in\bP^2(\overline{\bF}_q)$ with $P\in L$, we define
\begin{align*}
\mathcal{T}_{L, P} = 
\{ f\in R_{\homog} \ | \ C=\{f=0\} \text{ is tangent to } L \text{ at } P \}.
\end{align*}
Note that the condition that $C$ is tangent to $L$ at $P$ includes the possibility that $C$ is singular at $P\in L$. Similarly, we define,
\begin{align*}
    \mathcal{T}_{L} \colonequals \bigcup_{P\in L} \mathcal{T}_{L, P} = \{ f\in R_{\homog} \ | \ C=\{f=0\} \text{ is not transverse to } L\}
\end{align*} 
Note that polynomials in $\mathcal{T}_L$ correspond to those plane curves $C$ for which $L$ is tangent to $C$ somewhere, or $L$ passes through a singular point of $C$.

To compute the density of the set $\mathcal{T}_{L}$, we first need a precise count on the number of binary forms (homogeneous polynomials in two variables) that have a repeated root on $\bP^1$. In other words, we need to count the binary forms over $\bF_q$ that are \emph{not} square-free. Note that, for convenience, we view $0$ as a binary form in every degree, and $0$ is \textit{not} square-free.

\begin{lemma}\label{lemma:binary-forms-count}
Let $d\geq 3$. Let $\bF_{q}[x,y]_{d}$ denote the vector space of degree $d$ binary forms in variables $x$ and $y$. Consider the subset
\begin{align*}
    B_{d} = \{ f\in \bF_{q}[x,y]_{d} \ | \ f \text{ has a double root in } \bP^1(\overline{\bF}_q)\}.
\end{align*}
Then $\# B_{d} = q^{d}+q^{d-1}-q^{d-2}$.
\end{lemma}

\begin{proof}
Let us partition $B_{d}$ as
\begin{align*}
    B_{d} \setminus \{0\} = \bigsqcup_{i=0}^{d} B^{i}_{d}
\end{align*}
where
\begin{align*}
    B^{i}_{d} = \{ f\in B_d \ | \ f = a_i x^{d-i} y^{i} + a_{i+1} x^{d-i-1} y^{i+1} + \cdots +   a_d y^{d} \text{ with } a_i\neq 0 \}.
\end{align*}
It is well-known that the number of square-free monic polynomials of degree $d$ in $\bF_q[t]$ is $q^{d}-q^{d-1}$. Thus, the number of monic polynomials in $\bF_q[t]$ of degree $d$ with a repeated root is $q^{d-1}$. As a result, the number of binary forms,
\begin{align*}
    a_0 x^{d} + a_{1} x^{d-1} y + \cdots +   a_d y^{d}
\end{align*}
with $a_0\neq 0$ and with a repeated root is exactly $q^{d-1}\cdot (q-1)$, where the factor of $(q-1)$ is needed for scaling by $a_0\in \bF_{q}^{\ast}$. This proves that $\# B_d^{0}= q^{d}-q^{d-1}$. When $a_0=0$, we have
\begin{align*}
    a_{1} x^{d-1} y + \cdots +   a_d y^{d}
\end{align*}
If $a_1\neq 0$, we are looking at the elements of $B_{d}^{1}$. After factoring out $y$, we apply previous argument to get $\#B_{d}^{1}=q^{d-2}\cdot(q-1)=q^{d-1}-q^{d-2}$, where we used the assumption that $d\geq 3$. 

Note that if $a_0=0$ and $a_1=0$, then the binary form is not square-free for $d\geq 2$, because it would already have a factor of $y^2$. As a result, we have
\begin{align*}
    \# \left(\{0\}\cup \bigsqcup_{i=2}^{d} B_{d}^{i}\right) = q^{d-1}.
\end{align*}
Therefore for $d\geq 3$,
\begin{align*}
    \# B_d = \underbrace{(q^d-q^{d-1})}_{\# B_{d}^{0}}+\underbrace{(q^{d-1}-q^{d-2})}_{\# B_{d}^{1}}+q^{d-1}
    = q^{d}+q^{d-1}-q^{d-2}
\end{align*}
as desired.
\end{proof}

\begin{remark}
While have computed the exact formula for $\#B_d$ for each $d\geq 3$, we will only need this result for $d$ sufficiently large. For the sake of completeness, let us say a few words about the cases $d=1$ and $d=2$. We have $B_1=\{0\}$ and so $\#B_1=1$. Moreover, $\#B_2=q^2$ because $B_{2}$ consists of all quadratic forms $a(s x + ty)^2$ where $a\in\bF_q^{\ast}$ and $[s:t]\in\bP^1(\bF_q)$ together with $0$. This gives a total of $(q-1)(q+1)+1=q^2$ many quadratic forms that have a repeated root. 
\end{remark}

\begin{lemma}\label{lemma:tangent-to-line}
Let $L$ be an $\bF_q$-line. Then $\mu_d(\mathcal{T}_L) = q^{-1}+q^{-2}-q^{-3}$ for $d\geq 3$. In particular, $\mu(\mathcal{T}_L) = q^{-1}+q^{-2}-q^{-3}$.
\end{lemma}

\begin{proof} Without loss of generality, we may assume that $L=\{z=0\}$. Given a plane curve $C=\{F(x,y,z)=0\}$, the condition that $L$ is not transverse to $C$ is equivalent to the assertion that $F(x,y,0)$ is a binary form with a repeated root on $\mathbb{P}^1$. Consider the linear map
\begin{align*}
   \phi_d: \bF_{q}[x,y,z]_{d} &\to \bF_{q}[x,y]_{d} \\
    F &\mapsto F(x,y,0)
\end{align*}
 for each degree $d\geq 1$. Let $m=\dim \bF_q[x,y,z]_{d}-\dim\bF_{q}[x,y]_d = \binom{d+2}{2}-(d+1)$. Since $\phi_d$ is a surjective linear map, every fiber $\phi_d^{-1}(f)$ consists of $q^{m}$ many polynomials in $\bF_q[x,y,z]_{d}$.

Let $B_{d}\subset \bF_{q}[x,y]_{d}$ denote the subset consisting of binary forms of degree $d$ with a repeated root as in Lemma~\ref{lemma:binary-forms-count}. By the observation in the first paragraph,
\begin{align*}
    \mathcal{T}_{L} = \bigsqcup_{d\geq 1}\phi_d^{-1}(B_d).
\end{align*}
Using Lemma~\ref{lemma:binary-forms-count} we deduce that for $d\geq 3$,
\begin{align*}
\mu_{d}(\mathcal{T}_{L}) = \frac{\# \phi_d^{-1}(B_d)}{\# R_d} =\frac{\#B_d\cdot q^{m}}{\# \bF_q[x,y]_{d}\cdot q^m} = \frac{\# B_d}{\# \bF_q[x,y]_{d}}= \frac{q^{d}+q^{d-1}-q^{d-2}}{q^{d+1}}= \frac{1}{q}+\frac{1}{q^2}-\frac{1}{q^3}
\end{align*}
as desired.
\end{proof}

The following general result explains how to obtain the densities of $\mathcal{S}_Q$ and $\mathcal{T}_{L,P}$ and intersections of such sets for \textit{distinct} points. In particular, we get $\mu_d(\mathcal{S}_Q)=q^{-3}$ and $\mu_d(\mathcal{T}_{L, P})=q^{-2\deg(P)}$ for sufficiently large $d$.

\begin{lemma}\label{lemma:multi-line-independence}
Let $P_1,\ldots,P_n$ be a collection of distinct points, and $L_1,\ldots,L_n$ be a collection of $\bF_q$-lines (possibly with repetition) so that $P_i\in L_i$ for each $1\leq i\leq n$.  Furthermore, let $Q_1,\ldots,Q_s$ be a collection of $\bF_q$-points such that no $Q_j$ is on any $L_i$. Then, for $d$ sufficiently large, the events of being tangent to $L_i$ at $P_i$, and being singular at each $Q_j$ are independent. More precisely,
\begin{align*}
\mu_d\left(\bigcap_{i=1}^n \mathcal{T}_{L_i,P_i}\cap \bigcap_{j=1}^{s} \mathcal{S}_{Q_j}\right) = q^{-3s-\sum_{i=1}^n 2\deg(P_i)}
\end{align*}
for all sufficiently large $d$.
\end{lemma}

\begin{proof} For each $P_i$, consider the subscheme of $\bP^2$ defined by
\begin{align*}
    Y_{i} = \spec\left(\frac{\mathbb{F}_{q}[x,y,z]}{\mathfrak{m}_{P_i}^2}\right),
\end{align*}
where $\mathfrak{m}_{P_i}$ is the maximal ideal consisting of polynomials vanishing at $P_i$. Note that the support of $Y_i$ consists of only the point $P_i$, and the non-reduced structure on $Y_i$ keeps track of the first order tangent information. Similarly, for each $Q_j$, define
\begin{align*}
    Z_{j} = \spec\left(\frac{\mathbb{F}_{q}[x,y,z]}{\mathfrak{m}_{Q_j}^2}\right).
\end{align*}
By taking the disjoint union of these schemes, we define
\begin{align*}
Y = \left(\coprod_{i=1}^{n} Y_i \right) \coprod\left( \coprod_{j=1}^{s} Z_j \right).
\end{align*}
 The natural restriction map,
\begin{align*}
\phi_d: \underbrace{H^{0}(\mathbb{P}^2, \mathcal{O}_{\mathbb{P}^2}(d))}_{=R_{d}} \to H^{0}(Y, \mathcal{O}_{Y}(d)),
\end{align*}
is surjective for $d\geq \dim H^{0}(Y, \mathcal{O}_{Y}(d))-1$ by \cite{Poo04}*{Lemma 2.1}.

Next, there is a codimension $2\deg(P)$ linear space $W_{i}\subset H^{0}(Y_{i}, \mathcal{O}_{d})$ which parametrizes the condition that the tangent direction at $P_i$ is given by the line $L_i$.  Thus, the condition that the given curve $C=\{F=0\}$ is tangent to $L_{i}$ at the point $P_{i}$ for $1\leq i\leq n$ and is singular at $Q_{j}$ for $1\leq j\leq s$ is equivalent to asserting that
\begin{align*}
\phi_d(F)\in W_{1}\times W_{2}\times \cdots \times W_{n}\times \underbrace{\{0\}\times \{0\}\times\cdots \times \{0\}}_{s \text{ copies}} \subset H^{0}(Y, \mathcal{O}_{Y}(d)).
\end{align*}
Since $\prod_{i=1}^{n} W_i\times \prod_{j=1}^{s} \{0\}$ has codimension $3s+\sum_{i=1}^{n} 2\deg(P_i)$ inside the vector space $H^0(Y, \mathcal{O}_Y(d))$, and codimension is preserved under surjective linear maps, it follows that 
\begin{align*}
    \mu_d\left(\bigcap_{i=1}^{n} \mathcal{T}_{L, P_j}\cap \bigcap_{j=1}^{s} \mathcal{S}_{Q_j}\right) &= \frac{\#\phi_{d}^{-1}(\prod_{i=1}^{n} W_i\times \prod_{j=1}^{s} \{0\})}{\# R_d} \\
    &= \prod_{i=1}^{n} q^{-2\deg(P_i)}\cdot \prod_{j=1}^{s} q^{-3} \\
    &= q^{-3s-\sum_{i=1}^{n}2\deg(P_i)}
\end{align*}
for sufficiently large $d$. As an immediately corollary, we also get
\begin{align*}
\mu\left(\bigcap_{i=1}^{n} \mathcal{T}_{L, P_i}\cap \bigcap_{j=1}^{s} \mathcal{S}_{Q_j}\right) 
= \lim_{d\to\infty} \mu_d \left(\bigcap_{i=1}^{n} \mathcal{T}_{L, P_i}\cap \bigcap_{j=1}^{s} \mathcal{S}_{Q_j}\right)  = q^{-3s-\sum_{i=1}^{n}2\deg(P_i)}
\end{align*}
as desired.  \end{proof}

\begin{remark}\label{rmk:independence-may-fail}
If we impose tangency conditions to more than one line at the same point, these conditions are no longer independent. For instance if two lines $L_1,L_2$ meet at a point $Q$, then
\begin{align*}
\mathcal{T}_{L_1,Q}\cap \mathcal{T}_{L_2,Q} = \mathcal{S}_{Q}
\end{align*}
Now, using Lemma~\ref{lemma:multi-line-independence}, we have
\begin{align*}
q^{-3} = \mu(\mathcal{S}_Q) = \mu\big(T_{L_1,Q}\cap T_{L_2,Q}\big) \neq \mu(T_{L_1,Q})\cdot \mu(T_{L_2,Q}) = q^{-4}.
\end{align*}
\end{remark}

While Remark~\ref{rmk:independence-may-fail} shows that the exact independence cannot be guaranteed in general, Lemma~\ref{lemma:many-lines-one-point} below tells us that we can at least bound the density of the intersection by the product of the densities of individual events. 

Let $Q\in\bP^2(\bF_q)$ and let $L_1,\ldots, L_{q+1}$ be all the $q+1$ distinct $\bF_q$-lines containing $Q$. Let $\mathcal{A}_0(Q)$ be the collection of polynomials defining plane curves that are tangent to none of the $L_i$ at $Q$, namely
\begin{align*}
\mathcal{A}_0(Q) \colonequals R_{\homog} \backslash \left(\bigcup_{i=1}^{q+1} \mathcal{T}_{L_i,Q}\right).
\end{align*}
And for $1\leq j\leq q+1$, let $\mathcal{A}_j$ be the collection of curves tangent to $L_j$ at $Q$ but not tangent to any other $L_i$ at $Q$, namely
\begin{align*}
\mathcal{A}_{L_j}(Q) \colonequals \mathcal{T}_{L_j,Q} \backslash \left(\bigcup_{i\neq j} \mathcal{T}_{L_i,Q}\right).
\end{align*}
Recall that a curve is tangent to two different lines at $Q$ if and only if it is singular at $Q$. As a result, $\mathcal{A}_{L_j}(Q)$ is precisely equal to $\mathcal{T}_{L_j,Q}\setminus S_{Q}$.  In general, given an $\bF_q$-line $L$ passing through an $\bF_q$-point $Q$, we define
\begin{align*}
    \mathcal{A}_L(Q) \colonequals \mathcal{T}_{L,Q} \setminus \mathcal{S}_{Q}
\end{align*}

\begin{lemma}\label{lemma:many-lines-one-point}
Let $Q\in\bP^2(\bF_q)$. Consider all the $\bF_q$-lines $L_1, \ldots, L_{q+1}$ containing $Q$. Then, for sufficiently large $d$,
\begin{align}\label{eq:probability-of-A_0}
\mu_{d}(\mathcal{A}_0(Q)) \leq \frac{1}{1-q^{-2}}\prod_{i=1}^{q+1}\mu_{d}\big(R_{\homog}\backslash \mathcal{T}_{L_i,Q}\big),
\end{align}
where the product runs over all $\bF_q$-lines passing through $Q$. Moreover, for each $1\leq j\leq q+1$,
\begin{align}\label{eq:probability-of-A_j}
\mu_{d}(A_{L_j}(Q)) \leq \mu_{d}(\mathcal{T}_{L,Q})\prod_{i\neq j}\mu_{d}\big(R_{\homog}\backslash \mathcal{T}_{L_i,Q}\big).
\end{align}
Consequently, the inequalities \eqref{eq:probability-of-A_0} and \eqref{eq:probability-of-A_j} also hold with $\mu_d$ replaced by $\mu$.
\end{lemma}

\begin{proof} Let us begin with the set $\mathcal{A}_{L_j}$ for $1\leq j\leq q+1$. From the point of view of the point $Q$, the lines $L_{L_1},\ldots,L_{L_{q+1}}$ are indistinguishable, so $\mu(\mathcal{A}_{L_1})=\cdots=\mu(\mathcal{A}_{L_{q+1}})$. Hence we may focus on $\mathcal{A}_{L_1}$.

Since $\mathcal{A}_{L_1} = \mathcal{T}_{L_1,Q}\backslash \mathcal{S}_Q$ and $\mathcal{S}_Q\subset \mathcal{T}_{L_1,Q}$, we have that for sufficiently large $d$,
\begin{align}\label{eq:density-of-Aj}
\mu_d(\mathcal{A}_{L_1}(Q)) = \mu_d(\mathcal{T}_{L_1,Q}) - \mu_d(\mathcal{S}_Q) = q^{-2} - q^{-3},
\end{align}
using Lemma~\ref{lemma:multi-line-independence}.

Meanwhile, we have
\begin{align*}
\mu_d(\mathcal{T}_{L_1,Q})\prod_{i=2}^{q+1}\mu_d\big(R_{\homog}\backslash \mathcal{T}_{L_i,Q}\big) = q^{-2}(1-q^{-2})^{q}.
\end{align*}
for all sufficiently large $d$. Let us investigate the ratio,
\begin{align}\label{eq:ratio-for-A1}
\frac{\mu_{d}(\mathcal{T}_{L_1,Q})\displaystyle\prod_{i=2}^{q+1}\mu_{d}\big(R_{\homog}\backslash \mathcal{T}_{L_i,Q}\big)}{\mu_{d}(\mathcal{A}_1)} = \frac{q^{-2}(1-q^{-2})^{q}}{q^{-2}-q^{-3}} = \frac{(1-q^{-2})^{q}}{1-q^{-1}}.
\end{align}
We claim that
\begin{align}\label{eq:size-of-psi}
    \psi(q) \colonequals \frac{(1-q^{-2})^q}{1-q^{-1}}\geq 1
\end{align}
for each $q\geq 2$. We compute the derivative of $\psi$ with respect to $q$:
\begin{align*}
\psi'(q)  = \frac{(1-q^{-2})^q}{q^2-1}\Big(1+q(q+1)\log(1-q^{-2})\Big).
\end{align*}
The fraction $\frac{(1-q^{-2})^q}{q^2-1}$ is positive, while
\begin{align*}
1+q(q+1)\log(1-q^{-2}) \leq 1-q(q+1)\cdot q^{-2} = -q^{-1} < 0.
\end{align*}
Thus, the fraction $\psi(q)$ is decreasing with $q$. Moreover, it is straightforward to check that
\begin{align*}
\lim_{q\to\infty}\psi(q) = \lim_{q\to\infty}\frac{(1-q^{-2})^q}{1-q^{-1}} = 1.
\end{align*}
This completes the proof of \eqref{eq:size-of-psi}. Combining \eqref{eq:ratio-for-A1} and \eqref{eq:size-of-psi}, we deduce that
\begin{align*}
\frac{\displaystyle \mu_d(T_{L_1,Q})\prod_{i=2}^{q+1}\mu_d\big(R_{\homog}\backslash \mathcal{T}_{L_i,Q}\big)}{\mu_d(\mathcal{A}_1(Q))} = \frac{(1-q^{-2})^{q}}{1-q^{-1}} \geq 1.
\end{align*}
which rearranges into \eqref{eq:probability-of-A_j}. To analyze the density of $\mathcal{A}_0(Q)$, we first consider the complementary set $R_{\homog}\backslash \mathcal{A}_0(Q)$. A plane curve defined by an element of $R_{\homog}\backslash \mathcal{A}_0(Q)$ is either singular at $Q$, or else tangent to exactly one of the lines $L_1,\ldots,L_{q+1}$. In other words, $R_{\homog}\backslash\mathcal{A}_0(Q)$ can be partitioned into
\begin{align*}
    R_{\homog}\backslash\mathcal{A}_0(Q) = \mathcal{S}_{Q} \bigsqcup \left(\bigsqcup_{i=1}^{q+1} \mathcal{A}_{L_i}(Q)\right).
\end{align*}
Using Lemma~\ref{lemma:multi-line-independence} and \eqref{eq:density-of-Aj}, we deduce that
\begin{align*}
\mu_d(\mathcal{A}_0(Q)) = 1 - q^{-3} - (q+1)\mu_d(\mathcal{A}_{L_1}(Q)) = 1 - q^{-3} - (q+1)(q^{-2}-q^{-3}).
\end{align*}
for sufficiently large $d$. On the other hand,
\begin{align*}
\prod_{i=1}^{q+1} \mu_d\Big(R_{\homog}\backslash \mathcal{T}_{L_i,Q}\Big) = (1-q^{-2})^{q+1}.
\end{align*}
for sufficiently large $d$. We examine the ratio
\begin{align*}
\frac{\displaystyle\prod_{i=1}^{q+1}\mu_d\big(R_{\homog}\backslash \mathcal{T}_{L_i,Q}\big)}{\mu_d(\mathcal{A}_0(Q))} = \frac{(1-q^{-2})^{q+1}}{1-q^{-3} + (q+1)(q^{-3}-q^{-2})}
 & = \frac{(1-q^{-2})^{q+1}}{1-q^{-1}}.
\end{align*}
Using \eqref{eq:size-of-psi}, this quantity can be bounded below as follows.
\begin{align*}
\frac{(1-q^{-2})^{q+1}}{1-q^{-1}} = (1-q^{-2})\frac{(1-q^{-2})^{q}}{1-q^{-1}}\geq 1-q^{-2}.
\end{align*}
Hence for all $q\geq 2$, we have
\begin{align*}
\frac{\displaystyle\prod_{i=1}^{q+1}\mu_d\big(R_{\homog}\backslash \mathcal{T}_{L_i,Q}\big)}{\mu_d(\mathcal{A}_0(Q))} = \frac{(1-q^{-2})^{q+1}}{1-q^{-3} + (q+1)(q^{-3}-q^{-2})} \geq 1-q^{-2}.
\end{align*}
which justifies~\eqref{eq:probability-of-A_0}.
\end{proof}

Our final proposition in this section is an independence result that essentially states that the events $\mathcal{T}_{L, P}, \mathcal{S}_{Q}, \mathcal{A}_0(Q), \mathcal{A}_{L}(Q)$, and their complements are independent in sufficiently large degree.

\begin{lemma}\label{lemma:super-independence}
Consider any finite collection of sets $\mathcal{D}_j$, where each $\mathcal{D}_j$ is one of $\mathcal{T}_{L_j,P_j}$, $\mathcal{S}_{Q_j}$, $\mathcal{A}_0(Q_j)$, $\mathcal{A}_{L_j}(Q_j)$, or the complement of one of the above within $R_{\homog}$. Here, $P_i$ are arbitrary points of $\bP^2$, and $Q_j$ are arbitrary $\bF_q$-points of $\bP^2$. If the $P_i$'s and $Q_j$'s are all distinct, then for sufficiently large $d$,
\begin{align*}
    \mu_d\left(\bigcap \mathcal{D}_j\right) = \prod_j\mu_d(\mathcal{D}_j).
\end{align*}
\end{lemma}

\begin{proof}
First, note that if each $\mathcal{D}_j$ is either of the form $\mathcal{T}_{L_j,P_j}$ or $\mathcal{S}_{P_j}$ then the result was established in Lemma~\ref{lemma:multi-line-independence}. We first show that if
\begin{align*}
\mu_d\left(\bigcap_{j=1}^m \mathcal{D}_j\right) = \prod_{j=1}^m\mu_d(\mathcal{D}_j),
\end{align*}
holds for all subsets of $\{\mathcal{D}_1, \mathcal{D}_2, \mathcal{D}_{3}, \ldots \}$, then
\begin{align*}
\mu_d\left(\Big(R_{\homog}\backslash \mathcal{D}_1\Big)\cap\left(\bigcap_{j=2}^m \mathcal{D}_j\right)\right) = \mu_d(R_{\homog}\backslash \mathcal{D}_1)\prod_{j=2}^m \mu_d(\mathcal{D}_j).
\end{align*}
Since each $\mu_d$ is a probability measure on $R_d$, we have $\mu_d(R_{\homog}\backslash \mathcal{D}_1) = 1-\mu_d(\mathcal{D}_1)$. Writing $\bigcap_{j=2}^m \mathcal{D}_j = \left(\bigcap_{j=1}^m \mathcal{D}_j\right)\cup\left(\Big(R_{\homog}\backslash \mathcal{D}_1\Big)\cap\left(\bigcup_{j=2}^m  \mathcal{D}_j\right)\right)$ as a disjoint union, we have
\begin{align*}
\mu_d\left(\Big(R_{\homog}\backslash \mathcal{D}_1\Big)\cap\left(\bigcup_{j=2}^m \mathcal{D}_j\right)\right) &=  \mu\left(\bigcap_{j=2}^m \mathcal{D}_j\right) - \mu_d\left(\bigcap_{j=1}^m \mathcal{D}_j\right)\\
 &= \Big(1-\mu_d(\mathcal{D}_1)\Big)\mu_d\left(\bigcap_{j=2}^m \mathcal{D}_j\right)\\
 &= \mu_d(R_{\homog}\backslash \mathcal{D}_1)\prod_{j=2}^m\mu_d(\mathcal{D}_j).
\end{align*}
Inductively one may extend this procedure to any number of complements among the $\mathcal{D}_j$.

To deal with sets of the form $\mathcal{A}_{L_j}(Q_j)$, we will write $\mathcal{A}_{L_j}(Q_j) = \mathcal{T}_{L_j,Q_j}\backslash S_{Q_j}$. For example, if $\mathcal{D}_1 = \mathcal{A}_{L_1}(Q_1)$, and the remainder of the $\mathcal{D}_j$'s are of the form $\mathcal{D}_j=\mathcal{T}_{L_j,P_j}$ or $\mathcal{D}_j = \mathcal{S}_{Q_j}$, then write
\begin{align*}
\mu_d\left(\bigcap_{j=1}^m \mathcal{D}_j\right) &= \mu_d\left(\mathcal{T}_{L_1,Q_1}\cap\left(\bigcap_{j=2}^m \mathcal{D}_j\right)\right) - \mu_d\left(\mathcal{S}_{Q_1}\cap\left(\bigcap_{j=2}^m \mathcal{D}_j\right)\right)\\
 &= \mu_d(\mathcal{T}_{L_1,Q_1})\prod_{j=2}^m\mu_d(\mathcal{D}_j) - \mu_d(\mathcal{S}_{Q_1})\prod_{j=2}^m\mu_d(\mathcal{D}_j)\\
 &= \mu_d(\mathcal{A}_{L_1}(Q_1))\prod_{j=2}^m\mu_d(\mathcal{D}_j) = \prod_{j=1}^m\mu_d(\mathcal{D}_j).
\end{align*}
One can induct on the number of times sets of the form $\mathcal{A}_{L_j}(Q_j)$ appears on the list to allow any number of instances.

Finally, we explain how to deal with the case when $\mathcal{D}_j$ is of the form $\mathcal{A}_0(Q_j)$. As in the proof of Lemma~\ref{lemma:many-lines-one-point}, we can express
\begin{align*}
\mathcal{A}_0(Q_j) = R_{\homog}\backslash \left(\mathcal{S}_{Q_j}\cup\left(\bigsqcup_{Q_j\in L}\mathcal{A}_{L}(Q_j)\right)\right),
\end{align*}
where the union inside the parenthesis runs over all $\bF_q$-lines $L$ such that $Q_j\in L$. By the observation made in the beginning of the proof, it suffices to deal with the case when $\mathcal{D}_j$ is a complement of $\mathcal{A}_0(Q_j)$ for some $\bF_q$-point $Q_j$.  For example, assume that $\mathcal{D}_1 = R_{\homog}\backslash \mathcal{A}_0(Q_1)$, and the remainder of the $\mathcal{D}_j$ are of the other forms considered earlier. Since $\mathcal{D}_1$ can be written as a disjoint union, 
\begin{align*}
\mu_d\left(\bigcap_{j=1}^m \mathcal{D}_j\right) &= \mu_d\left(\left(\mathcal{S}_{Q_1}\cap \bigcap_{j=2}^m \mathcal{D}_j \right)\cup \bigsqcup_{Q_1\in L} \left(\mathcal{A}_{L}(Q_1)\cap \bigcap_{j=2}^m \mathcal{D}_j \right) \right) \\
&=\mu_d\left(\mathcal{S}_{Q_1}\cap \bigcap_{j=2}^m \mathcal{D}_j \right) +\sum_{Q_1\in L} \mu_d\left(\mathcal{A}_{L}(Q_1)\cap \bigcap_{j=2}^m \mathcal{D}_j\right) \\
&=\left(\mu_d(\mathcal{S}_{Q_1}) + \sum_{Q_1\in L}\mu_d(\mathcal{A}_{L}(Q_1))\right) \prod_{j=2}^m \mu_d(\mathcal{D}_j)\\
 &= \mu_d(\mathcal{D}_1)\cdot \prod_{j=2}^m \mu_d(\mathcal{D}_j) =\prod_{j=1}^m\mu_d(\mathcal{D}_j).
\end{align*}
As before, inducting on the number of times a set of the form $\mathcal{A}_0(Q_j)$ appears in the list allows for any number of instances.
\end{proof}

\section{Proof of the upper bound}\label{section:upper-bound}

In this section, our goal is to prove the following result.

\begin{theorem}\label{thm:precise-upper-bound} Let $\mathcal{F}\subset R=\bF_q[x,y,z]$ denote the subset defining smooth transverse-free plane curves. Then
\begin{align*}
    \overline{\mu}(\mathcal{F})\leq (1-q^{-2})^{-(q^2+q+1)} \cdot \left(\frac{1}{q}+\frac{1}{q^2}-\frac{1}{q^3}\right)^{q^2+q+1}.
\end{align*}
\end{theorem}
Let us explain how the upper bound in Theorem~\ref{thm:precise-upper-bound} is stronger than the upper bound given in Theorem~\ref{thm:smooth-proportion-transverse-free}. It suffices to show that
\begin{align}\label{eq:inequality-7.5}
    (1-q^{-2})^{-(q^2+q+1)}\leq 7.5 
\end{align}
for all $q\geq 2$. The function $\xi(q)\colonequals (1-q^{-2})^{-(q^2+q+1)}$ is decreasing for $q\geq 2$, and in fact, $\displaystyle\lim_{q\to\infty} \xi(q)=e$. In particular, we have $\xi(q)\leq \xi(2)$. Consequently,
\begin{align*}
    \xi(q) =  (1-q^{-2})^{-(q^2+q+1)} \leq 
    \xi(2) = \left(\frac{4}{3}\right)^7 \approx 7.4915409 < 7.5
\end{align*}
justifying \eqref{eq:inequality-7.5}.

\subsection{Reduction to the bounded degree case.}\label{subsect:reduction} Our ultimate goal is to understand the upper density of $\mathcal{F}$. By definition, $f\in \mathcal{F}$ defines a smooth plane curve $C$ such that all $\bF_q$-lines are tangent to $C$. If $C=\{f=0\}$ and $L$ is an $\bF_q$-line tangent to $C$, let $P$ be a point of tangency. The only control on the degree of $P$ is $\deg(P)\leq d/2$ by B\'ezout's theorem. In this subsection, we will explain a reduction argument that allows us to consider the case where all the points of tangencies satisfy $\deg(P)\leq r$ for some fixed integer $r>0$.

Let $L_1, L_2, \ldots, L_{q^2+q+1}$ be the $\bF_q$-lines in $\bP^2$. For each $1\leq i\leq q^2+q+1$, consider the set
\begin{align*}
\mathcal{M}^{i}_{t} = \{f\in R_{d} \ | \ & C=\{f=0\} \text{ is tangent to the line } L_i \\ &\text{at a point } P\in L_i \text{ with } t\leq \deg P\leq d/2 
\}.
\end{align*}
By Poonen's result \cite{Poo04}*{Lemma 2.4} on the medium degree singularities, we have
\begin{align}\label{eq:medium-degree-zero-measure}
    \lim_{t\to \infty} \overline{\mu}(\mathcal{M}^{i}_{t}) = 0.
\end{align}
Let $\varepsilon>0$. By \eqref{eq:medium-degree-zero-measure}, there exists a large enough integer $r=r(\varepsilon)>0$ such that
\begin{align*}
     \overline{\mu}(\mathcal{M}^{i}_{t}) < \frac{\varepsilon}{q^2+q+1}
\end{align*}
for all $t>r$. Note that the value of $r$ does not actually depend on $i$, because of the symmetry among the $\bF_q$-lines. As a result,
\begin{align}\label{eq:medium-degree-epsilon}
\overline{\mu}\left(\bigcup_{i=1}^{q^2+q+1}\mathcal{M}^{i}_{t}\right) < \varepsilon
\end{align} 
for all $t>r$. We can summarize \eqref{eq:medium-degree-epsilon} as stating that there exists a positive integer $r$ (depending on $\varepsilon$) so that the collection of curves with any tangency along an $\bF_q$-line at a point of degree $t>r$ has natural density less than $\varepsilon$. Now, let us partition the set $\mathcal{F}$ in Theorem~\ref{thm:smooth-proportion-transverse-free} as
\begin{align}\label{eq:decomposition-of-F}
    \mathcal{F} = \mathcal{F}_{\leq r} \cup \left(\mathcal{F} \setminus \mathcal{F}_{\leq r}\right)
\end{align}
where
\begin{align*}
\mathcal{F}_{\leq r}\colonequals \{f\in \mathcal{F} \ | \ & C=\{f=0\} \text{ is a smooth plane curve such} \\ 
&\text{that all its tangency points along $\bF_q$-lines}  \\ 
&\text{have degree}  \leq r \}.
\end{align*}
Combining \eqref{eq:medium-degree-epsilon} with the inclusion
\begin{align*}
    \mathcal{F} \setminus \mathcal{F}_{\leq r} \subset \bigcup_{i=1}^{q^2+q+1}\mathcal{M}^{i}_{r+1} 
\end{align*}
we obtain $\overline{\mu}(\mathcal{F} \setminus \mathcal{F}_{\leq r})<\varepsilon$. Consequently, the partition in~\eqref{eq:decomposition-of-F} yields
\begin{align}\label{eq:density-of-F-and-Fr}
    \overline{\mu}(\mathcal{F}) < \overline{\mu}(\mathcal{F}_{\leq r}) + \varepsilon.
\end{align}
To prove Theorem~\ref{thm:precise-upper-bound}, it suffices to show the following inequality.
\begin{align}\label{eq:upper-density-of-Fr}
     \overline{\mu}(\mathcal{F}_{\leq r})\leq (1-q^{-2})^{-(q^2+q+1)} \cdot \left(q^{-1}+q^{-2}-q^{-3}\right)^{q^2+q+1}
\end{align}
Indeed, combining \eqref{eq:density-of-F-and-Fr} and \eqref{eq:upper-density-of-Fr} and letting $\varepsilon\to 0$, we obtain the conclusion of Theorem~\ref{thm:precise-upper-bound}. Thus, our goal has been reduced to proving the following theorem.

\begin{theorem}\label{thm:precise-upper-bound-degree-r} Let $\mathcal{F}_{\leq r}\subset R=\bF_q[x,y,z]$ denote the subset defining smooth transverse-free plane curves, where the tangency point(s) along each $\bF_q$-line have degree at most $r$. Then
\begin{align*}
    \overline{\mu}(\mathcal{F}_{\leq r})\leq (1-q^{-2})^{-(q^2+q+1)} \cdot\left(q^{-1}+q^{-2}-q^{-3}\right)^{q^2+q+1}.
\end{align*}
\end{theorem}

\subsection{The key decomposition.}  As a preparation for the proof of Theorem~\ref{thm:precise-upper-bound-degree-r}, we explain how to include $\mathcal{F}_{\leq r}$ into a finite union of sets indexed by the location of tangency points. The advantage of doing so is that the density of each piece will be easier to compute. A finer version of this decomposition will be used again in Section~\ref{sect:mu-exists}.

\begin{notation}
Given an $\bF_q$-line $L$, let $L^r$ denote the set of points on $L$ of degree $\leq r$.
\end{notation}

Let $L_1, L_2, \ldots, L_{q^2+q+1}$ be all of the $\bF_q$-lines in $\bP^2$. Given a line $L_i$ and a set $E_i\subseteq L_i^r$, consider the collection $\mathcal{A}_{E_i}$ of curves tangent to $L_i$ at all the points of $E_i$ and at no other points of $L_i^r$. In symbols,
\begin{align*}
\mathcal{A}_{E_i} \colonequals \left(\bigcap_{P\in E_i}\mathcal{T}_{L_i,P}\right)\backslash\left(\bigcup_{P\in L_i^r\backslash E_i}\mathcal{T}_{L_i,P}\right) = \bigcap_{P\in L_i^r}\begin{cases}\mathcal{T}_{L_i,P} & P\in E_i\\R_{\homog}\backslash \mathcal{T}_{L_i,P} & P\not\in E_i\end{cases}.
\end{align*}

Observe that if $f\in \mathcal{F}_{\leq r}$ then for each $\bF_q$-line $L_i$ there is a non-empty set $E_i\subseteq L_i^r$ of points where the curve $C=\{f=0\}$ is tangent to $L_i$. Moreover, by smoothness assumption on $C$, no point can appear in more than one $E_i$. Thus, we may write
\begin{align}\label{eq:decomposition-of-Fr}
\mathcal{F}_{\leq r} \subseteq \bigcup_{\substack{\emptyset\neq E_i\subseteq L_i^r, \ 1\leq i\leq q^2+q+1\\ E_i\cap E_j=\emptyset, \ i\neq j}}\left(\bigcap_{i=1}^{q^2+q+1} \mathcal{A}_{E_i}\right).
\end{align}

\subsection{Proof of Theorem~\ref{thm:precise-upper-bound-degree-r}} Using the decomposition given in \eqref{eq:decomposition-of-Fr}, we are now ready to prove Theorem~\ref{thm:precise-upper-bound-degree-r}. First, given an $\bF_q$-point $Q$, let 
\begin{align*}
    \mathcal{A}(Q) \colonequals \bigcap_{Q\in L_i}\begin{cases}\mathcal{T}_{L_i,Q} & Q\in E_i\\R_{\homog}\backslash \mathcal{T}_{L_i,Q} & Q\not\in E_i\end{cases}
\end{align*}
This is consistent with the notation used in Lemma~\ref{lemma:many-lines-one-point}. Indeed, $\mathcal{A}(Q)$ is either $\mathcal{A}_{0}(Q)$ or $\mathcal{A}_{L}(Q)$ for a unique $\bF_q$-line $L$ containing $Q$.  Then
\begin{eqnarray*}
\bigcap_{i=1}^{q^2+q+1} \mathcal{A}_{E_i} & = & \bigcap_{i=1}^{q^2+q+1} \bigcap_{P\in L_i^r}\begin{cases} \mathcal{T}_{L_i,P} & P\in E_i\\R_{\homog}\backslash \mathcal{T}_{L_i,P} & P\not\in E_i\end{cases}\\
 & = & \left(\bigcap_{Q\in \bP^2(\bF_q)} \mathcal{A}(Q)\right)\cap\left(\bigcap_{i=1}^{q^2+q+1} \ \bigcap_{P\in L_i^r\backslash \bP^2(\bF_q)}\begin{cases}\mathcal{T}_{L_i,P} & P\in E_i\\R_{\homog}\backslash \mathcal{T}_{L_i,P} & P\not\in E_i\end{cases}\right).
\end{eqnarray*}
By Lemma~\ref{lemma:super-independence}, for $d$ sufficiently large we may write
\begin{align}\label{eq:density-of-intersection-of-all-Ei}
\mu_d\left(\bigcap_{i=1}^{q^2+q+1} \mathcal{A}_{E_i}\right) = \left(\prod_{Q\in\bP^2(\bF_q)}\mu_{d}(\mathcal{A}(Q))\right)\prod_{i=1}^{q^2+q+1}\prod_{P\in L_i^r\backslash\bP^2(\bF_q)}\begin{cases}\mu_d(\mathcal{T}_{L_i,P}) & P\in E_i\\\mu_d(R_{\homog}\backslash \mathcal{T}_{L_i,P}) & P\not\in E_i\end{cases}.
\end{align}
We will now use Lemma~\ref{lemma:many-lines-one-point} to give an upper bound for each $\mu_{d}(\mathcal{A}(Q))$, which will provide an upper bound for \eqref{eq:density-of-intersection-of-all-Ei}. If $Q\in E_k$ for some $k$, then $Q\not\in E_j$ for all $j\neq k$ by the hypothesis. In this case, we have $\mathcal{A}(Q)=\mathcal{A}_{L_k}(Q)$ and Lemma~\ref{lemma:many-lines-one-point} yields
\begin{align}\label{eq:density-of-A_L(Q)}
   \mu_d(\mathcal{A}(Q))=\mu_d(\mathcal{A}_{L_k}(Q)) \leq \prod_{i=1}^{q^2+q+1} \begin{cases}\mu_d(\mathcal{T}_{L_i,Q}) & Q\in E_i\\\mu_d(R_{\homog}\backslash \mathcal{T}_{L_i,Q}) & Q\not\in E_i\end{cases}.
\end{align}
If $Q\notin E_i$ for all $1\leq i\leq q^2+q+1$, then $\mathcal{A}(Q)=\mathcal{A}_{0}(Q)$ and Lemma~\ref{lemma:many-lines-one-point} yields,
\begin{align}\label{eq:density-of-A_0(Q)}
   \mu_d(\mathcal{A}(Q))=\mu_d(\mathcal{A}_{0}(Q)) \leq \frac{1}{1-q^{-2}} \prod_{i=1}^{q^2+q+1} \begin{cases}\mu_d(\mathcal{T}_{L_i,Q}) & Q\in E_i\\\mu_d(R_{\homog}\backslash \mathcal{T}_{L_i,Q}) & Q\not\in E_i\end{cases}.
\end{align}
Substituting \eqref{eq:density-of-A_L(Q)} and \eqref{eq:density-of-A_0(Q)} into  \eqref{eq:density-of-intersection-of-all-Ei}, we obtain the following upper bound:
\begin{align*}
    \mu_d\left(\bigcap_{i=1}^{q^2+q+1} \mathcal{A}_{E_i}\right) \leq  \left(\frac{1}{1-q^{-2}}\right)^{\#\left(\bP^2(\bF_q)\backslash\bigcup_{i=1}^{q^2+q+1}E_i\right)}\prod_{i=1}^{q^2+q+1}\prod_{P\in L_i^r}\begin{cases}\mu_d(\mathcal{T}_{L_i,P}) & P\in E_i\\\mu_d(R_{\homog}\backslash \mathcal{T}_{L_i,P}) & P\not\in E_i\end{cases}.
\end{align*}
After bounding the exponent of $(1-q^{-2})^{-1}$ by $\#\bP^2(\bF_q) = q^2+q+1$ and using Lemma~\ref{lemma:multi-line-independence}, we deduce that
\begin{align*}
    \mu_d\left(\bigcap_{i=1}^{q^2+q+1} \mathcal{A}_{E_i}\right) \leq \left(\frac{1}{1-q^{-2}}\right)^{q^2+q+1}\left(\prod_{i=1}^{q^2+q+1}\left(\prod_{P\in E_i}q^{-2\deg(P)}\prod_{P\in L_i^r\backslash E_i}(1-q^{-2\deg(P)})\right)\right)
\end{align*}
holds for $d$ sufficiently large. Now, we use the previous inequality to bound the union in \eqref{eq:decomposition-of-Fr}.
\begin{align*}
\mu_d(\mathcal{F}_{\leq r}) &\leq \left(\frac{1}{1-q^{-2}}\right)^{q^2+q+1}\cdot \sum_{\substack{\emptyset\neq E_i\subseteq L_i^r \\ E_i\cap E_j=\emptyset,\ i\neq j}}\left(\prod_{i=1}^{q^2+q+1}\left(\prod_{P\in E_i}q^{-2\deg(P)}\prod_{P\in L_i^r\backslash E_i}(1-q^{-2\deg(P)})\right)\right)\\
 &\leq \left(\frac{1}{1-q^{-2}}\right)^{q^2+q+1} \ \cdot \ \prod_{i=1}^{q^2+q+1}\left(\sum_{\emptyset\neq E_i\subseteq L_i^r}\left(\prod_{P\in E_i}q^{-2\deg(P)}\prod_{P\in L_i^r\backslash E_{i}}(1-q^{-2\deg(P)})\right)\right)\\
 &=  \Big(1-q^{-2}\Big)^{-(q^2+q+1)}\cdot \left(\sum_{\emptyset\neq E_1\subseteq L_1^r}\left(\prod_{P\in E_1}q^{-2\deg(P)}\prod_{P\in L_1^r\backslash E_1}(1-q^{-2\deg(P)})\right)\right)^{q^2+q+1}.
\end{align*}
Going from the first to the second line, we added terms where some points appear in more than one $E_i$, which makes the sum bigger. Going from the second to the third line, we used the symmetry among the $\bF_q$-lines $L_1, \ldots, L_{q^2+q+1}$.

Finally, since the sum on the last line represents all the ways a curve can be tangent to the line $L_1$ at a point of degree at most $r$, we apply Lemma~\ref{lemma:tangent-to-line} to bound this sum by $q^{-1}+q^{-2}-q^{-3}$. Indeed, as $r\to\infty$, the sum approaches $q^{-1}+q^{-2}-q^{-3}$ from below. We deduce that
\begin{align*}
\mu_d(\mathcal{F}_{\leq r}) \leq (1-q^{-2})^{-(q^2+q+1)}\cdot (q^{-1}+q^{-2}-q^{-3})^{q^2+q+1}
\end{align*}
for all sufficiently large $d$. Taking the limit supremum as $d\to\infty$, we obtain 
\begin{align*}
\overline{\mu}(\mathcal{F}_{\leq r}) \leq (1-q^{-2})^{-(q^2+q+1)}\cdot (q^{-1}+q^{-2}-q^{-3})^{q^2+q+1}
\end{align*} 
as desired. 

This completes the proof of Theorem~\ref{thm:precise-upper-bound-degree-r}. As explained in Section~\ref{subsect:reduction}, this implies Theorem~\ref{thm:precise-upper-bound}.

\section{Existence of the natural density} \label{sect:mu-exists}

In this section, we prove that $\mu(\mathcal{F})$ exists as a limit. As a consequence,  $\underline{\mu}(\mathcal{F})=\overline{\mu}(\mathcal{F})=\mu(\mathcal{F})$. 

\begin{proposition}\label{prop:mu-exists} The subset
\begin{align*}
\mathcal{F}=\{f\in R_{\homog} \ | \ & C=\{f=0\} \text{ is a smooth plane curve} \\ 
&\text{such that all $\bF_q$-lines are tangent to $C$}\}
\end{align*}
has a well-defined natural density $\mu(\mathcal{F})$.
\end{proposition}

\begin{proof} As explained in Section~\ref{subsect:reduction}, we can write
\begin{align*}
    \mathcal{F} = \bigcup_{r=1}^{\infty} \mathcal{F}_{\leq r}.
\end{align*}
We will show that $\mu(\mathcal{F}_{\leq r})$ exists for each $r\geq 1$. Let us first explain how this will prove the proposition. The sequence $(\mu(\mathcal{F}_{\leq r}))_{r=1}^{\infty}$ is clearly increasing and also bounded by Theorem~\ref{thm:precise-upper-bound-degree-r}, and hence converges. Using inequality~\eqref{eq:density-of-F-and-Fr}, for each $\varepsilon>0$ there exists a positive integer $r(\varepsilon)$ such that for all $r>r(\epsilon)$
\begin{align*}
    \mu(\mathcal{F}_{\leq r}) \leq \underline{\mu}(\mathcal{F}) \leq \overline{\mu}(\mathcal{F}) \leq \mu(\mathcal{F}_{\leq r})+\varepsilon. \end{align*}
Letting $r\to\infty$,
\begin{align*}
    \lim_{r\to\infty}\mu(\mathcal{F}_{\leq r}) \leq \underline{\mu}(\mathcal{F}) \leq \overline{\mu}(\mathcal{F}) \leq \lim_{r\to\infty}\mu(\mathcal{F}_{\leq r})+\varepsilon. \end{align*}
As observed above, the limit $\displaystyle \lim_{r\to\infty}\mu(\mathcal{F}_{\leq r})$ exists. Finally, letting $\varepsilon\to 0$ shows that $\mu(\mathcal{F})$ exists and equals $\displaystyle\lim_{r\to\infty}\mu(\mathcal{F}_{\leq r})$.
    
In order to show that $\mu(\mathcal{F}_{\leq r})$ exists, we will decompose $\mathcal{F}_{\leq r}$ as a finite union similar to the inclusion in \eqref{eq:decomposition-of-Fr} but with an equality. As before, let $L_1, \ldots, L_{q^2+q+1}$ be the $\bF_q$-lines in $\bP^2$. For each subset $E_i\subset L_i$, let
\begin{align*}
    \mathcal{A}^{\operatorname{sm}}_{E_i} =\{f\in R_{\homog} \ | \ & C=\{f=0\} \text{ is a smooth plane curve such that } C \text{ is tangent to } \\ 
& L_i \text{ at each point of } E_i \text{ and at no point of } L_i\setminus E_i \}.
\end{align*} 
Then we have the following decomposition
\begin{align}\label{eq:exact-decomposition-of-Fr}
\mathcal{F}_{\leq r} = \bigcup_{\substack{\emptyset\neq E_i\subseteq L_i^r, \ 1\leq i\leq q^2+q+1\\ E_i\cap E_j=\emptyset, \ i\neq j}}\left(\bigcap_{i=1}^{q^2+q+1} \mathcal{A}^{\operatorname{sm}}_{E_i}\right).
\end{align}
Since the union in~\eqref{eq:exact-decomposition-of-Fr} is a disjoint union,
\begin{align}\label{eq:exact-decomposition-of-Fr-mu-d}
\mu_d\left(\mathcal{F}_{\leq r}\right) = \sum_{\substack{\emptyset\neq E_i\subseteq L_i^r, \ 1\leq i\leq q^2+q+1\\ E_i\cap E_j=\emptyset, \ i\neq j}} \mu_d\left(\bigcap_{i=1}^{q^2+q+1} \mathcal{A}^{\operatorname{sm}}_{E_i}\right).
\end{align}
If we can show that the limit
\begin{align}\label{eq:natural-density-of-AEis}
    \mu\left(\bigcap_{i=1}^{q^2+q+1} \mathcal{A}^{\operatorname{sm}}_{E_i}\right) = \lim_{d\to\infty} \mu_d\left(\bigcap_{i=1}^{q^2+q+1} \mathcal{A}^{\operatorname{sm}}_{E_i}\right)
\end{align}
exists, then we can take the limit as $d\to\infty$ in \eqref{eq:exact-decomposition-of-Fr-mu-d} which will show that $\mu(\mathcal{F}_{\leq r}) = \displaystyle \lim_{d\to\infty} \mu_d(\mathcal{F}_{\leq r})$ also exists. 

In order to prove the limit in \eqref{eq:natural-density-of-AEis} exists, we appeal to Poonen's Bertini theorem with Taylor conditions \cite{Poo04}*{Theorem 1.2}. The proof here is similar to those of Lemma~\ref{lemma:lower-bound-for-single-point-line} and Lemma~\ref{lemma:multi-line-independence}. Since $r$ is fixed and each $E_i$ is finite, membership in the intersection of $\mathcal{A}_{E_i}^{\operatorname{sm}}$ is described by finitely many local conditions. In other words, there is a finite subscheme $Z\subset\bP^2$ and a finite subset $T\subset H^{0}(Z, \mathcal{O}_{Z})$ such that a given polynomial $f\in R_d$ satisfies 
\begin{align*}
    f\in  \bigcap_{i=1}^{q^2+q+1} \mathcal{A}^{\operatorname{sm}}_{E_i}
\end{align*}
if and only if $\phi_d(f)\in T$, where
\begin{align*}
\phi_d: \underbrace{H^{0}(\mathbb{P}^2, \mathcal{O}_{\mathbb{P}^2}(d))}_{=R_{d}} \to H^{0}(Z, \mathcal{O}_{Z}(d))
\end{align*}
is the natural restriction map. By \cite{Poo04}*{Theorem 1.2}, it follows that
\begin{align*}
    \mu\left(\bigcap_{i=1}^{q^2+q+1} \mathcal{A}^{\operatorname{sm}}_{E_i}\right) = \frac{\# T}{\# H^{0}(Z, \mathcal{O}_Z)} \zeta_{\bP^2}(3)^{-1}, 
    \end{align*}
and in particular the natural density in \eqref{eq:natural-density-of-AEis} exists. This completes the proof of the proposition.
\end{proof}

Finally, we prove our main result. 

\begin{proof}[Proof of Theorem~\ref{thm:smooth-proportion-transverse-free}] Combine Theorem~\ref{thm:lower-bound}, Theorem~\ref{thm:precise-upper-bound} and Proposition~\ref{prop:mu-exists}.
\end{proof}

\bibliographystyle{alpha}
\bibliography{proportion-bibliography.bib}

\end{document}